\documentclass[english,12pt,oneside]{amsproc}
\usepackage[T2A]{fontenc}
\usepackage[english]{babel}
\usepackage{sseq}
\usepackage{graphics}
\usepackage{amsfonts, amssymb, amscd, amsmath}
\usepackage{latexsym}
\usepackage[matrix,arrow,curve]{xy}
\usepackage{mathabx,mathtools}
\usepackage{color}
\usepackage{mathrsfs}
\usepackage{mathdots}
\usepackage{pigpen}
\usepackage{tikz}
\usetikzlibrary{matrix}

\DeclareMathOperator{\lk}{lk} \DeclareMathOperator{\cone}{Cone}
 
 \DeclareMathOperator{\id}{id}

 \DeclareMathOperator{\im}{Im}
 
\DeclareMathOperator{\ver}{Vert}

\newcommand{\Zo}{\mathbb{Z}}
\newcommand{\Ro}{\mathbb{R}}
\newcommand{\Co}{\mathbb{C}}
\newcommand{\Qo}{\mathbb{Q}}

\newcommand{\ko}{\Bbbk}

\newcommand{\ang}{\raisebox{-2pt}{\pigpenfont G}}
\newcommand{\simc}{\!\!\sim}
\newcommand{\be}{\beta}
\newcommand{\br}{\widetilde{\beta}}
\newcommand{\ft}{\widehat{f}}
\newcommand{\chir}{\widetilde{\chi}}

\newcommand{\eqd}{\stackrel{\text{\tiny def}}{=}}

\newcommand{\minel}{\hat{0}}

\newcommand{\sta}[1]{(\ast_{#1})}

\newcommand{\E}[1]{(E_{#1})}
\newcommand{\dif}[1]{(d_{#1})}

\newcommand{\difm}[1]{d_{#1}^{-}}

\newcommand{\Ea}[1]{(\dot{E}_{#1})}
\newcommand{\difa}[1]{(\dot{d}_{#1})}
\newcommand{\fa}{\dot{f}}
\newcommand{\Ead}[2]{(\dot{E}^{#2}_{#1})}

\newcommand{\Hr}{\widetilde{H}}
\newcommand{\dd}{\partial}

\newcounter{stmcounter}[section]
\newcounter{thcounter}

\numberwithin{equation}{section}

\theoremstyle{plain}
\newtheorem{cor}[stmcounter]{Corollary}
\newtheorem{stm}[stmcounter]{Statement}
\newtheorem{thm}[thcounter]{Theorem}
\newtheorem{prop}[stmcounter]{Proposition}
\newtheorem{lemma}[stmcounter]{Lemma}
\newtheorem{defin}[stmcounter]{Definition}

\newtheorem{claim}[stmcounter]{Claim}

\theoremstyle{definition}

\newtheorem{rem}[stmcounter]{Remark}
\newtheorem{con}[stmcounter]{Construction}

\begin{document}

\title[Locally standard torus actions and $h'$-vectors]{Locally
standard torus actions and $h'$-vectors of simplicial posets}

\author[Anton Ayzenberg]{Anton Ayzenberg}
\address{Department of Mathematics, Osaka City University, Sumiyoshi-ku, Osaka 558-8585, Japan.}
\email{ayzenberga@gmail.com}

\date{\today}
\thanks{The author is supported by the JSPS postdoctoral fellowship program.}

\begin{abstract}
We consider the orbit type filtration on a manifold $X$ with
locally standard action of a compact torus and the corresponding
homological spectral sequence $\E{X}^r_{*,*}$. If all proper faces
of the orbit space $Q=X/T$ are acyclic and free part of the action
is trivial, this spectral sequence can be described in full.
The ranks of diagonal terms are equal to the $h'$-numbers of the
Buchsbaum simplicial poset $S_Q$ dual to $Q$. Betti numbers of $X$
depend only on the orbit space $Q$ but not on the characteristic
function. If $X$ is a slightly different object, namely the model
space $X=(P\times T^n)/~\simc$ where $P$ is a cone over Buchsbaum
simplicial poset $S$, we prove that $\dim \E{X}^{\infty}_{p,p} =
h''_p(S)$. This gives a topological evidence for the fact that
$h''$-numbers of Buchsbaum simplicial posets are nonnegative.
\end{abstract}

\maketitle


%
%
%
%
%
%
%

\section{Introduction}\label{SecIntro}

An action of a compact torus $T^n$ on a smooth compact manifold
$M$ of dimension $2n$ is called locally standard if it is locally
modeled by the standard representation of $T^n$ on $\Co^n$. The
orbit space $Q=M/T^n$ is a manifold with corners. Every manifold
with locally standard torus action is equivariantly homeomorphic
to the quotient construction $X=Y/\simc$, where $Y$ is a principal
$T^n$-bundle over $Q$ and $\sim$ is an equivalence relation given
by a characteristic function on $Q$ (see \cite{Yo}).
%

In the case when all faces of the orbit space (including $Q$
itself) are acyclic, Masuda and Panov \cite{MasPan} proved that
$H^*_T(M;\Zo)\cong \Zo[S_Q]$ and $H^*(M;\Zo)\cong
\Zo[S_Q]/(l.s.o.p)$, where $S_Q$ is a simplicial poset dual to
$Q$; $\Zo[S_Q]$ is the face ring with even grading; and
$(l.s.o.p)$ is the system of parameters of degree $2$ determined
by the characteristic function. In this situation $S_Q$ is a
Cohen--Macaulay simplicial poset, so $(l.s.o.p)$ is actually a
regular sequence in the ring $\Zo[S_Q]$. In particular this
implies $\dim H^{2j}(M)=h_j(S_Q)$ and $H^{2j+1}(M)=0$.

These considerations generalize similar results for quasitoric
manifolds, complete smooth toric varieties, and symplectic toric
manifolds which were known before. One can see that there are many
examples of manifolds $M$ whose orbit spaces are acyclic.
Nevertheless, several constructions had appeared in the last years
providing natural examples of manifolds with nontrivial topology
of the orbit space. These constructions include for example toric
origami manifolds \cite{SGP} and toric log symplectic manifolds
\cite{GLPR}.

It seems that the most reasonable assumption which is weaker than
acyclicity of all faces, but still allows for explicit
calculations is the following. We assume that every \emph{proper}
face of $Q$ is acyclic, and $Y$ is a trivial $T^n$-bundle:
$Y=Q\times T^n$. This paper is the second in a series of works,
where we study homological structure of $M$ under these assumption
by using the orbit type filtration. The general scopes of this
work are described in the preprint \cite{AyTot}.

It is convenient to work with a quotient construction $X=(Q\times
T^n)/\simc$ instead of~$M$. The orbit type filtration $X_0\subset
X_1\subset \ldots\subset X_n$ covers the natural filtration
$Q_0\subset Q_1\subset \ldots\subset Q_n$ of $Q$, and is covered
by a filtration $Y_0\subset Y_1\subset \ldots\subset Y_n$ of $Y$,
where $Y_i=Q_i\times T^n$. In the previous paper \cite{AyV1} we
proved that homological spectral sequences associated with
filtrations on $Y$ and $X$ are closely related. Namely, there is
an isomorphism of the second pages $f_*^2\colon \E{Y}^2_{p,q}\to
\E{Y}^2_{p,q}$ for $p>q$, when $Q$ has acyclic proper faces.

In this paper we calculate the ranks of groups in the spectral
sequence and Betti numbers of $X$. Since $Y=Q\times T^n$, the
spectral sequence $\E{Y}^*_{*,*}$ is isomorphic to
$\E{Q}^*_{*,*}\otimes H_*(T^n)$. The structure of $\E{Q}^*_{*,*}$
can be explicitly described. This is done in
Section~\ref{secTopolQ}. As a technical tool, we introduce the
modified spectral sequence $\Ea{Q}^*$ which coincides with
$\E{Q}^*$ from the second page, and whose first page
$\Ea{Q}^1_{*,*}$ in a certain sense lies in between
$\E{Q}^1_{*,*}$ and $\E{Q}^2_{*,*}$. Similar constructions of
modified spectral sequences $\Ea{Y}^*$ and $\Ea{X}^*$ are
introduced for $Y$ and $X$ in Section \ref{SecTorusActions}.

The induced map $\fa_*^1\colon \Ea{Y}^1_{p,q}\to \Ea{X}^1_{p,q}$
is an isomorphism for $p>q$, as follows essentially from the
result of \cite{AyV1}. This gives a description of all
differentials and all non-diagonal terms of $\Ea{X}^1_{*,*}$,
which is stated in Theorem \ref{thmEplus1struct}. The diagonal
terms of the spectral sequence require an independent
investigation. We prove, in particular, that $\dim \E{X}^2_{q,q} =
\dim \Ea{X}^2_{q,q} = h'_{n-q}(S_Q)$
--- the $h'$-number of the dual simplicial poset
(Theorem~\ref{thmBorderManif}). The proof involves combinatorial
computations and is placed in separate Section \ref{secHvectors},
where we give all necessary definitions from the combinatorial
theory of simplicial posets. The appearance of $h'$-vector in this
problem is quite natural. If $Q$ has acyclic proper faces, the
dual simplicial poset $S_Q$ is Buchsbaum. Recall that $h'$-vector
is a combinatorial notion, which was specially devised to study
the combinatorics of Buchsbaum simplicial complexes.

In Section~\ref{secHomology} we introduce the bigraded structure
on $H_*(X;\ko)$ and compute bigraded Betti numbers (Theorem
\ref{thmHomolX}). Bigraded Poincare duality easily follows from
this computation.

Most of the arguments used for manifolds with locally standard
actions, work equally well for the space $X=(P\times T^n)/\simc$
where $P$ is the cone over Buchsbaum simplicial poset equipped
with the dual face structure. In this case there holds $\dim
\E{X}^{\infty}_{q,q}=h''_q(S)$ (Theorem~\ref{thmHtwoPrimes}). Over
rational numbers, every simplicial poset admits a characteristic
function, therefore our theorem implies $h''_q(S)\geqslant 0$ for
a Buchsbaum simplicial poset $S$. This result was proved by Novik
and Swartz in \cite{NS} by a different method.

Both cases, manifolds with acyclic proper faces and cones over
Buchsbaum posets, are unified in the notion of Buchsbaum
pseudo-cell complex, introduced in Section \ref{SecPreliminaries}.
The technique developed in the paper can be applied to any
Buchsbaum pseudo-cell complex.

In the last section we analyze a simple example which shows that
without the assumption of proper face acyclicity the problem of
computing Betti numbers of $X$ is more complicated. In general,
Betti numbers of $X$ may depend not only on the orbit space $Q$,
but also on the characteristic function.

%
%
%
%
%
%
%

\section{Preliminaries}\label{SecPreliminaries}

\subsection{Coskeleton filtrations and manifolds with corners.}

\begin{defin}\label{definSimpPoset}
A finite partially ordered set (poset in the following) is called
simplicial if there is a minimal element $\minel\in S$ and, for
any $I\in S$, the lower order ideal $\{J\in S\mid J\leqslant I\}$
is isomorphic to the poset of faces of a $(k-1)$-simplex, for some
$k\geqslant 0$.
\end{defin}

The elements of $S$ are called simplices. The number $k$ in the
definition is denoted $|I|$ and called the rank of $I$. Also set
$\dim I = |I|-1$. A simplex of rank $1$ is called a vertex; the
set of all vertices is denoted $\ver(S)$. The \emph{link} of a
simplex $I\in S$ is the set $\lk_SI=\{J\in S\mid J\geqslant I\}$.
This set inherits the order relation from $S$, and $\lk_SI$ is a
simplicial poset with respect to this order, with $I$ being the
minimal element. Let $S'$ denote the barycentric subdivision of
$S$. By definition, $S'$ is a simplicial complex on the vertex set
$S\setminus\{\minel\}$ whose simplices are the ordered chains in
$S\setminus\{\minel\}$. The geometric realization of $S$ is the
geometric realization of its barycentric subdivision
$|S|\eqd|S'|$. One can also think of $|S|$ as a CW-complex with
simplicial cells (such complexes were called \emph{simplicial cell
complexes} in \cite{BPposets}). A poset $S$ is called \emph{pure}
if all its maximal elements have equal dimensions. A poset $S$ is
pure whenever $S'$ is pure.

Let $\ko$ denote a ground ring, which may be either $\Zo$ or a
field. The term ``(co)homology of simplicial poset'' means the
(co)homology of its geometrical realization. If the coefficient
ring in the notation of (co)homology is omitted, it is supposed to
be $\ko$. The rank of a $\ko$-module $A$ is denoted $\dim A$.

\begin{defin}\label{definBuchCMposets}
A simplicial poset $S$ of dimension $n-1$ is called Buchsbaum
(over $\ko$) if $\Hr_i(\lk_SI;\ko)=0$ for all $\minel\neq I\in S$
and $i\neq n-1-|I|$. If $S$ is Buchsbaum and, moreover,
$\Hr_i(S;\ko)=0$ for $i\neq n-1$, then $S$ is called
Cohen--Macaulay (over~$\ko$).
\end{defin}

By abuse of terminology we call $S$ a \emph{homology manifold} of
dimension $n-1$ if its geometric realization $|S|$ is a homology
$(n-1)$-manifold. Simplicial poset $S$ is Buchsbaum if and only if
it is Buchsbaum and, moreover, its local homology stack of highest
degree is isomorphic to a constant sheaf (see details in
\cite{AyV1}).

If $S$ is Buchsbaum and connected, then $S$ is pure. In the
following we consider only pure posets, and assume $\dim S=n-1$.

\begin{con}
For any pure simplicial poset $S$, there is an associated space
$P(S) = \cone|S|$ endowed with the dual face structure (also
called coskeleton structure), defined as follows. The complex
$P(S)$ is a simplicial complex on the set $S$ and $k$-simplices of
$S'$ have the form $(I_0< I_1<\ldots<I_k)$, where $I_j\in S$. For
each $I\in S$ consider the subsets:
\[
G_I=|\{(I_0< I_1<\ldots)\in S'\mbox{ such that } I_0\geqslant
I\}|\subset P(S),
\]
\[
\dd G_I=|\{(I_0< I_1<\ldots)\in S'\mbox{ such that } I_0>I\}|
\subset P(S).
\]
and the subset $G_I^{\circ}=G_I\setminus\dd G_I$. We have
$G_{\minel} = P(S)$; $G_I\subset G_J$ whenever $J<I$, and $\dim
G_I=n-1-\dim I$ since $S$ is pure. A subset $G_I$ is called a
\emph{dual face} of a simplex $I\in S$. A subset $\dd G_I$ is a
union of faces of smaller dimensions.
\end{con}

Recall several facts about manifolds with corners. A smooth
connected manifold with corners $Q$ is called \emph{nice} (or a
\emph{manifold with faces}) if every codimension $k$ face lies in
exactly $k$ distinct facets. In the following we consider only
nice compact orientable manifolds with corners. Any such $Q$
determines a simplicial poset $S_Q$ whose elements are the faces
of $Q$ ordered by reversed inclusion. The whole $Q$ is the maximal
face of itself, thus represents the minimal element of $S_Q$.

\begin{defin}
A nice manifold with corners $Q$ is called Buchsbaum if $Q$ is
orientable and every proper face of $Q$ is acyclic. If, moreover,
$Q$ is acyclic itself, it is called Cohen--Macaulay.
\end{defin}

If $Q$ is a Buchsbaum manifold with corners, then its underlying
simplicial poset $S_Q$ is Buchsbaum (moreover, $S_Q$ is a homology
manifold), and when $Q$ is Cohen--Macaulay, then so is $S_Q$
(moreover, $S_Q$ is a homology sphere) by \cite[Lm.6.2]{AyV1}.

\subsection{Buchsbaum pseudo-cell complexes}

It is convenient to introduce a notion which captures both
manifolds with corners and cones over simplicial posets.

\begin{con}[Pseudo-cell complex]
A CW-pair $(F,\dd F)$ will be called $k$-dimensional pseudo-cell,
if $F$ is compact and connected, $\dim F = k$, $\dim \dd
F\leqslant k-1$. A (regular finite) \emph{pseudo-cell complex} $Q$
is a space which is a union of an expanding sequence of subspaces
$Q_k$ such that $Q_{-1}$ is empty and $Q_{k}$ is the pushout
obtained from $Q_{k-1}$ by attaching finite number of
$k$-dimensional pseudo-cells $(F,\dd F)$ along injective attaching
CW-maps $\dd F\to Q_{k-1}$. We also assume that the boundary of
each pseudo-cell is a union of lower dimensional pseudo-cells. The
poset of pseudo-cells, ordered by the reversed inclusion is
denoted by $S_Q$. The abstract elements of $S_Q$ are denoted by
$I,J$, etc. and the corresponding pseudo-cells considered as
subsets of $Q$ are denoted $F_I,F_J$, etc.

A pseudo-cell complex $Q$, of dimension $n$ is called
\emph{simple} if $S_Q$ is a simplicial poset of dimension $n-1$
and $\dim F_I=n-1-\dim I$ for all $I\in S_Q$. In particular, the
space $Q$ itself represents the maximal pseudo-cell,
$Q=F_{\minel}$. Pseudo-cells of a simple pseudo-cell complex $Q$
will be called \emph{faces}, faces different from $Q$ ---
\emph{proper faces}, and maximal proper faces --- \emph{facets}.
Facets correspond to vertices of $S_Q$.
\end{con}

Examples of simple pseudo-cell complexes are nice manifolds with
corners and cones over simplicial posets. Simple polytopes are
examples, which lie in both of these classes.

\begin{defin}\label{definBuchPCcomplex}
A simple pseudo-cell complex $Q$ is called Buchsbaum (over $\ko$)
if, for any proper face $F_I\subset Q$, $I\neq\minel$, the
following conditions hold:
\begin{enumerate}
\item $F_I$ is acyclic, $\Hr_*(F_I;\ko)=0$;
\item $H_j(F_I,\dd F_I;\ko)=0$ for each $j\neq\dim F_I$.
\end{enumerate}
Buchsbaum complex $Q$ is called Cohen--Macaulay (over $\ko$) if
these two conditions also hold for the maximal face
$F_{\minel}=Q$.
\end{defin}

Both Buchsbaum manifolds and cones over Buchsbaum posets are
examples of Buchsbaum pseudo-cell complexes (and the same for
Cohen--Macaulay property). Indeed, in the case of Buchsbaum
manifold with corners, $\Hr_*(F_I)$ vanishes by definition and
$H_*(F_I,\dd F_I)$ vanishes in the required degrees by the
Poincare--Lefschetz duality, since every face $F_I$ is an
orientable manifold with boundary. In the cone case, we have
$G_I=\cone(\dd G_I)$ and $\dd G_I\cong |\lk_SI|$, so the
conditions of Definition~\ref{definBuchPCcomplex} follow from the
isomorphism $H_*(G_I,\dd G_I)\cong H_{*-1}(\dd G_I)$ and
Definition~\ref{definBuchCMposets}.

We have a topological filtration
\[
Q_0\subset Q_1\subset \ldots\subset Q_{n-1}\subset Q_n=Q;
\]
and a truncated filtration
\[
Q_0\subset Q_1\subset \ldots\subset Q_{n-1}=\dd Q,
\]
where $Q_j$ is a union of faces of dimension $\leqslant j$. The
homological spectral sequences associated with these filtrations
are denoted $\E{Q}^r_{p,q}$ and $\E{\dd Q}^r_{p,q}$ respectively.
The same argument as in \cite[Lm.6.2]{AyV1} proves the following

\begin{prop}\label{propUnivers}\mbox{}
\begin{enumerate}
\item Let $Q$ be a Buchsbaum pseudo-cell complex, $S_Q$ be
its underlying poset, and $P=P(S_Q)$ be the cone complex. Then
there exists a face-preserving map $\varphi\colon Q\to P$ which
induces the identity isomorphism of posets of faces and an
isomorphism of the truncated spectral sequences $\varphi_*\colon
\E{\dd Q}^r_{*,*}\stackrel{\cong}{\to}\E{\dd P}^r_{*,*}$ for
$r\geqslant 1$. In particular, if $Q$ is a Buchsbaum pseudo-cell
complex, then $S_Q$ is a Buchsbaum simplicial poset.

\item If $Q$ is a Cohen--Macaulay pseudo-cell complex of dimension $n$, then $\varphi$
induces an isomorphism of non-truncated spectral sequences
$\varphi_*\colon \E{Q}^r_{*,*}\stackrel{\cong}{\to}\E{P}^r_{*,*}$
for $r\geqslant 1$. In particular, if $Q$ is a Cohen--Macaulay
pseudo-cell complex, then $S_Q$ is a Cohen--Macaulay simplicial
poset.
\end{enumerate}
\end{prop}

Thus all homological information about Buchsbaum pseudo-cell
complex $Q$ away from its maximal cell is encoded in the
underlying poset $S_Q$. This makes Buchsbaum pseudo-cell
complexes, and in particular Buchsbaum manifolds with corners, a
good family to study.

%
%
%
%
%
%
%

\section{Spectral sequence of $Q$}\label{secTopolQ}

\subsection{Truncated and non-truncated spectral sequences}

In Buchsbaum case the spectral sequence $\E{Q}^r_{p,q}$ can be
described explicitly. We have $\E{Q}^r_{p,q}\Rightarrow
H_{p+q}(Q)$, the differentials act as $\dif{Q}^r\colon
\E{Q}^r_{p,q}\to \E{Q}^r_{p-r,q+r-1}$, and
\[
\E{Q}^1_{p,q}\cong H_{p+q}(Q_p,Q_{p-1})\cong \bigoplus_{I, \dim
F_I=p}H_{p+q}(F_I,\dd F_I).
\]
By the definition of Buchsbaum pseudo-cell complex, we have
$\E{Q}^1_{p,q}=0$ unless $q=0$ or $p=n$. Such form of the spectral
sequence will be referred to as \ang-shaped.

By forgetting the last term of the filtration we get the spectral
sequence $\E{\dd Q}^r_{p,q}\Rightarrow H_{p+q}(\dd Q)$, whose
terms vanish unless $q=0$. Thus $\E{\dd Q}^r_{p,q}$ collapses at a
second page, giving the isomorphism $\E{\dd Q}^2_{p,0}\cong
H_p(\dd Q)$.

In the non-truncated case we have $\E{Q}^2_{p,0}\cong \E{\dd
Q}^2_{p,0}$ for $p\neq n, n-1$. The terms $\E{Q}^2_{n,q}$ coincide
with $\E{Q}^1_{n,q}\cong H_{n+q}(Q,\dd Q)$ when $q\neq 0$. The
term $\E{Q}^2_{n-1,0}$ differs from $\E{\dd Q}^2_{n-1,0}\cong
H_{n-1}(\dd Q)$ by the image of the first differential $\dif{Q}^1$
which hit it at the previous step. Similarly, the term
$\E{Q}^2_{n,0}$ is the kernel of the same differential. To avoid
mentioning these two exceptional cases every time in the
following, we introduce the formalism of modified spectral
sequence.

\subsection{Modified spectral sequence}\label{subsecArtPage}

Let $\Ea{Q}^1_{*,*}$ be the collection of $\ko$-modules defined by
\[
\Ea{Q}^1_{p,q}\eqd\begin{cases} \E{\dd Q}^2_{p,q}, \mbox{ if } p\leqslant n-1,\\
\E{Q}^1_{p,q},\mbox{ if } p=n,\\
0,\mbox{ otherwise}.
\end{cases}
\]
Let $\difm{Q}$ be the differential of degree $(-1,0)$ acting on
$\bigoplus\E{Q}^1_{p,q}$ by:
\[
\difm{Q}=\begin{cases}\dif{Q}^1\colon
\E{Q}^1_{p,q}\to\E{Q}^1_{p-1,q}, \mbox{ if } p\leqslant n-1,\\0,
\mbox{ otherwise}
\end{cases}
\]
It is easily seen that the homology module $H(\E{Q}^1;\difm{Q})$
is isomorphic to $\Ea{Q}^1$. Now consider the differential
$\difa{Q}^1$ of degree $(-1,0)$ acting on
$\bigoplus\Ea{Q}^1_{p,q}$:
\[
\difa{Q}^1=\begin{cases} 0, \mbox{ if }p\leqslant n-1;\\
\E{Q}^1_{n,q}\stackrel{\dif{Q}^1}{\longrightarrow}\E{Q}^1_{n-1,q},\mbox{
if }p=n.
\end{cases}
\]
In the latter case, the image of the differential lies in
$\Ea{Q}^1_{n-1,q}\subseteq \E{Q}^{1}_{n-1,q}$ since
$\Ea{Q}^1_{n-1,q}$ is just the kernel of $\dif{Q}^1$. We have
$\E{Q}^2\cong H(\Ea{Q}^1,\difa{Q}^1)$. These considerations are
shown on the diagram:
\[
\xymatrix{
&\Ea{Q}^1\ar@{..>}[rd]^{\difa{Q}^1}&&&\\
\E{Q}^1\ar@{..>}[rr]^{\dif{Q}^1}\ar@{..>}[ru]^{\difm{Q}}&&
\E{Q}^2\ar@{..>}[r]^{\dif{Q}^2}&\E{Q}^3\ar@{..>}[r]^{\dif{Q}^3}&\ldots
}
\]
in which the dotted arrows represent passing to homology. To
summarize:

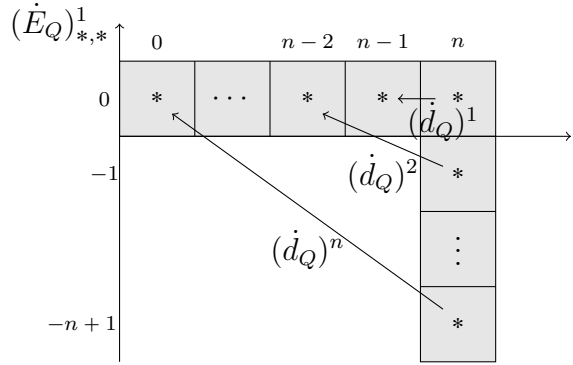
\begin{figure}[h]
\begin{center}
\begin{tikzpicture}[scale=1]
        \filldraw[fill=black!10] (0,0)--(0,1)--(5,1)--(5,-3)--(4,-3)--(4,0)--cycle;
        \draw[->]  (0,0)--(6,0);
        \draw[->]  (0,-3)--(0,1.5);
        \draw (1,0)--(1,1); \draw (2,0)--(2,1); \draw
        (3,0)--(3,1); \draw (4,0)--(4,1); \draw (4,0)--(5,0); \draw
        (4,-1)--(5,-1); \draw (4,-2)--(5,-2);

        \draw (0.5,0.5) node{$\ast$}; \draw (1.5,0.5)
        node{$\cdots$}; \draw (2.5,0.5) node{$\ast$}; \draw (3.5,0.5)
        node{$\ast$}; \draw (4.5,0.5) node{$\ast$}; \draw (4.5,-0.5)
        node{$\ast$}; \draw (4.5,-1.4) node{$\vdots$}; \draw (4.5,-2.5) node{$\ast$};

        \draw[->] (4.2,0.5)--(3.7,0.5); \draw[->]
        (4.3,-0.4)--(2.7,0.3); \draw[->] (4.3,-2.3)--(0.7,0.3);

        \draw (-0.8,1.5) node{$\Ea{Q}^1_{*,*}$}; \draw (0.5,1.25) node{\tiny
        $0$}; \draw (2.5,1.25) node{\tiny $n-2$}; \draw (3.5,1.25) node{\tiny
        $n-1$}; \draw (4.5,1.25) node{\tiny $n$};

        \draw (-0.2,0.5) node{\tiny $0$}; \draw (-0.2,-0.5) node{\tiny
        $-1$}; \draw (-0.5,-2.5) node{\tiny $-n+1$};

        \draw (4.3,0.2) node{$\difa{Q}^1$}; \draw (3.5,-0.5)
        node{$\difa{Q}^2$}; \draw (2.5,-1.5) node{$\difa{Q}^n$};
\end{tikzpicture}
\end{center}
\caption{The shape of the spectral sequence.} \label{figAngShaped}
\end{figure}

\begin{claim}\label{claimFormalPage}
There is a homological spectral sequence
$\Ea{Q}^r_{p,q}\Rightarrow H_{p+q}(Q)$ such that $\Ea{Q}^1_{*,*}=
H(\E{Q}^1,\difm{Q})$, and $\Ea{Q}^r_{*,*}=\E{Q}^r_{*,*}$ for
$r\geqslant 2$. The only nontrivial differentials of this sequence
have the form
\[
\difa{Q}^r\colon \Ea{Q}^r_{n,1-r}\to \Ea{Q}^r_{n-r,0}
\]
for $r\geqslant 1$ (see Figure \ref{figAngShaped}).
\end{claim}

The differentials have pairwise distinct domains and targets. Thus
the whole spectral sequence $\E{Q}^r_{p,q}\Rightarrow H_{p+q}(Q)$
folds into a single long exact sequence, which is isomorphic to a
long exact sequence of the pair $(Q,\dd Q)$:

\begin{equation}
\xymatrix{\ldots\ar@{->}[r]& H_{n+1-r}(Q)\ar@{=}[dd]\ar@{->}[r]&
\Ea{Q}_{n,1-r}^r\ar@{->}[r]^{\difa{Q}^r}&
\Ea{Q}_{n-r,0}^r\ar@{->}[r]&
H_{n-r}(Q)\ar@{->}[r]\ar@{=}[dd]&\ldots\\
&&\Ea{Q}^1_{n,1-r}\ar@{=}[u]\ar@{=}[d]&\Ea{Q}^1_{n-r,0}\ar@{->}[d]^{\cong}\ar@{=}[u]&&\\
\ldots\ar@{->}[r]& H_{n+1-r}(Q)\ar@{->}[r]&H_{n+1-r}(Q,\dd Q)
\ar@{->}[r]^(0.55){\delta_{n+1-r}}& H_{n-r}(\dd
Q)\ar@{->}[r]&H_{n-r}(Q)\ar@{->}[r]&\ldots}
\end{equation}

In particular, the differentials $\dif{Q}^r\colon
\E{Q}^r_{n,1-r}\to \E{Q}^r_{n-r,0}$ coincide up to isomorphism
with the connecting homomorphisms
\[
\delta_{n+1-r}\colon H_{n+1-r}(Q,\dd Q)\to H_{n-r}(\dd Q).
\]
This proves
\begin{prop}\label{propSpecSeqQForm}
Up to isomorphism, the spectral sequence
$\Ea{Q}^r_{*,*}\Rightarrow H_*(Q)$ has the form:
\begin{center}
\begin{tikzpicture}[scale=2]
        \filldraw[fill=black!10] (0,0)--(0,0.7)--(6,0.7)--(6,-2.1)--(4.5,-2.1)--(4.5,0)--cycle;
        \draw[->, thick, black]  (0,0)--(6.5,0);
        \draw[->, thick, black]  (0,-2)--(0,1.1);
        \draw (1,0)--(1,0.7); \draw (2,0)--(2,0.7); \draw
        (3,0)--(3,0.7); \draw (4.5,0)--(4.5,0.7); \draw (4.5,0)--(6,0); \draw
        (4.5,-0.7)--(6,-0.7); \draw (4.5,-1.4)--(6,-1.4);

        \draw (0.5,0.35) node{$H_0(\dd Q)$}; \draw (1.5,0.35)
        node{$\cdots$}; \draw (2.5,0.35) node{$H_{n-2}(\dd Q)$}; \draw (3.5,0.35) node{$H_{n-1}(\dd
        Q)$}; \draw (5.35,0.35) node{$H_n(Q,\dd Q)$}; \draw (5.25,-0.35) node{$H_{n-1}(Q,\dd
        Q)$}; \draw (5.25,-1) node{$\vdots$}; \draw (5.25,-1.75) node{$H_1(Q,\dd Q)$};

        \draw[->,thick] (4.7,0.35)--(4.1,0.35); \draw[->,thick]
        (4.6,-0.35)--(2.8,0.15); \draw[->,thick] (4.7,-1.75)--(0.8,0.15);

        \draw (0.4,1.1) node{$\Ea{Q}^1_{*,*}$}; \draw (0.5,0.85) node{\tiny
        $0$}; \draw (2.5,0.85) node{\tiny $n-2$}; \draw (3.75,0.85) node{\tiny
        $n-1$}; \draw (5.25,0.85) node{\tiny $n$};

        \draw (-0.15,0.35) node{\tiny $0$}; \draw (-0.15,-0.35) node{\tiny
        $-1$}; \draw (-0.27,-1.75) node{\tiny $-n+1$};

        \draw (4.4,0.55) node{$\difa{Q}^1=\delta_n$}; \draw (3.5,-0.3)
        node{$\difa{Q}^2=\delta_{n-1}$}; \draw (2.5,-1.1) node{$\difa{Q}^n=\delta_1$};
\end{tikzpicture}
\end{center}

\end{prop}

%
%
%
%
%
%
%

\section{Quotient construction and its spectral sequence}\label{SecTorusActions}

\subsection{Quotient construction}

Let $T^n$ denote a compact torus, and $\Lambda_*$ be its homology
algebra, $\Lambda_*=\bigoplus_{j=0}^n\Lambda_j$,
$\Lambda_j=H_j(T^n;\ko)$. Let $Q$ be a simple pseudo-cell complex
of dimension $n$, and $S_Q$ its dual simplicial poset. The map
\[
\lambda\colon \ver(S_Q)\to \{\mbox{1-dimensional toric subgroups
of } T^n\}
\]
is called \emph{characteristic function} if the following so
called $(\ast)$-condition holds: whenever $i_1,\ldots,i_k$ are the
vertices of some simplex in $S_Q$, the map
\begin{equation}\label{eqInclusionSubtori}
\lambda(i_1)\times\ldots\times\lambda(i_k)\to T^n,
\end{equation}
induced by inclusions $\lambda(i_j)\hookrightarrow T^n$, is
injective and splits. Note that $i_1,\ldots,i_k$ are the vertices
of some simplex, if and only if $F_{i_1}\cap\ldots\cap
F_{i_k}\neq\emptyset$. Denote the image of the map
\eqref{eqInclusionSubtori} by $T_I$, where $I$ is a simplex with
vertices $i_1,\ldots,i_k$.

It follows from the $(\ast)$-condition that the map
\begin{equation}\label{eqHomolSplits}
H_1(\lambda(F_1)\times\ldots\times\lambda(F_k);\ko)\to
H_1(T^n;\ko)
\end{equation}
is injective and splits for every $\ko$. If the map
\eqref{eqHomolSplits} splits for a specific ground ring $\ko$, we
say that $\lambda$ satisfies $\sta{\ko}$-condition and call it a
$\ko$-characteristic function. It is easy to see that the
topological $(*)$-condition is equivalent to $\sta{\Zo}$, and that
$\sta{\Zo}$ implies $\sta{\ko}$ for any $\ko$.

For a simple pseudo-cell complex $Q$ of dimension $n$, consider
the space $Y=Q\times T^n$.

\begin{con}
For any $\ko$-characteristic function $\lambda$ over $Q$ consider
the quotient construction
\[
X=Y/\simc=(Q\times T^n)/\simc,
\]
where $(q_1,t_1)\sim (q_2,t_2)$ if and only if $q_1=q_2\in
F_I^{\circ}$ for some $I\in S_Q$ and $t_1t_2^{-1}\in T_I$. The
action of $T^n$ on the second coordinate of $Y$ descends to the
action on $X$. The orbit space of this action is $Q$. The
stabilizer of the point $q\in F_I^{\circ}\subset Q$ is $T_I$. Let
$f$ denote the canonical quotient map from $Y$ to $X$.
\end{con}

The filtration on $Q$ induces filtrations on $Y$ and $X$:
\[
Y_i=Q_i\times T^n,\qquad X_i=Y_i/\simc,\quad i=0,\ldots,n.
\]
The filtration $X_0\subset X_1\subset\ldots\subset X_n=X$
coincides with the orbit type filtration. This means that $X_i$ is
a union of all torus orbits of dimension at most $i$. We have
$\dim X_i=2i$. We will use the following notation
\[
Y_I=F_I\times T^n,\qquad \dd Y_I=(\dd F_I)\times T^n
\]
\[
X_I=Y_I/\simc,\qquad \dd X_I = \dd Y_I/\simc
\]
for $I\in S_Q$. Note that $\dd X_I$ does not have the meaning of
topological boundary of $X_I$, this is just a conventional
notation. Since $\dd Q = Q_{n-1}$, we have $\dd Y = Y_{n-1} = (\dd
Q)\times T^n$ and $\dd X = X_{n-1} = \dd Y/\simc$.

There are homological spectral sequences
\[
\E{Y}^r_{p,q}\Rightarrow H_{p+q}(Y)\qquad \E{X}^r_{p,q}\Rightarrow
H_{p+q}(X)
\]
\[
\E{\dd Y}^r_{p,q}\Rightarrow H_{p+q}(\dd Y)\qquad \E{\dd
X}^r_{p,q}\Rightarrow H_{p+q}(\dd X),
\]
associated with these filtrations. The canonical map $f\colon Y\to
X$ induces the morphisms $f_*^r\colon \E{Y}^r_{*,*}\to
\E{X}^r_{*,*}$ and $f_*^r\colon \E{\dd Y}_{*,*}^r\to \E{\dd
X}_{*,*}^r$.

Since homology groups of the torus are torsion free, we have
\begin{equation*}
\E{Y}^r_{p,q}\cong \bigoplus_{q_1+q_2=q}\E{Q}^r_{p,q_1}\otimes
\Lambda_{q_2},
\end{equation*}
for $r\geqslant 1$ by Kunneth's formula. Similar for the truncated
spectral sequence:
\begin{equation*}
\E{\dd Y}^r_{p,q}\cong \bigoplus_{q_1+q_2=q}\E{\dd
Q}^r_{p,q_1}\otimes \Lambda_{q_2},
\end{equation*}

\subsection{Modified spectral sequences}

As in the case of $Q$ and absolutely similar to that case, we
introduce the modified spectral sequences $\Ea{Y}_{*,*}$ and
$\Ea{X}_{*,*}$. Consider the bigraded module:
\[
\Ea{Y}^1_{p,q}=\begin{cases} \E{\dd Y}_{p,q}^2, \mbox{ if }p<n;\\
\E{Y}^1_{n,q}, \mbox{ if }p=n.
\end{cases}
\]
and define the differentials $\difm{Y}\colon\E{Y}^1_{p,q}\to
\E{Y}^1_{p-1,q}$ and $\difa{Y}^1\colon\Ea{Y}^1_{p,q}\to
\Ea{Y}^1_{p-1,q}$ by
\[
\difm{Y}=\begin{cases} \dif{Y}^1, \mbox{ if }p<n;\\
0,\mbox{ if }p=n.
\end{cases}\qquad
\difa{Y}^1=\begin{cases} 0, \mbox{ if }p<n;\\
\E{Y}^1_{n,q}\stackrel{\dif{Y}^1}{\longrightarrow}\E{Y}^1_{n-1,q}
\mbox{ if }p=n.
\end{cases}
\]
It is easily checked that $\Ea{Y}^1\cong H(\E{Y}^1,\difm{Y})$ and
$\E{Y}^2\cong H(\Ea{Y}^1,\difa{Y}^1)$. Let $\Ea{Y}^r=\E{Y}^r$ for
$r\geqslant 2$. Thus we have the modified spectral sequence
$\Ea{Y}^r_{*,*}\Rightarrow H_*(Y)$. For $r\geqslant 1$ we have
\begin{equation}\label{eqKunnethYSpecSec}
\Ea{Y}^r_{p,q}\cong \bigoplus_{q_1+q_2=q}\Ea{Q}^r_{p,q_1}\otimes
\Lambda_{q_2},
\end{equation}

The same construction applies for $X$, thus we get the spectral
sequence $\Ea{X}^r_{*,*}\Rightarrow H_*(X)$ such that
$\Ea{X}^1\cong H(\E{X}^1,\difm{X})$, and $\Ea{X}^r=\E{X}^r$ for
$r\geqslant 2$. There exists an induced map of the modified
spectral sequences:
\[
\fa^r_*\colon \Ea{Y}^r \to \Ea{X}^r.
\]

By dimensional reasons the homological spectral sequence
$\E{X}^r_{p,q}\Rightarrow H_{p+q}(X)$ (and therefore its modified
version) has an obvious vanishing property:
\[
\E{X}^1_{p,q}=H_{p+q}(X_p,X_{p-1}) = 0 \mbox{ for } q>p.
\]

\begin{figure}[h]
\begin{center}
    \begin{tikzpicture}[scale=.6]


        \filldraw[fill=black!40]
        (0,0)--(0,6)--(6,6)--(6,0)--cycle;
        \filldraw[fill=black!25]
        (0,0)--(0,1)--(1,1)--(1,2)--(2,2)--(2,3)--(3,3)--(3,4)--(4,4)--(4,5)--(5,5)--(5,6)--(6,6)--(6,-4)--(5,-4)--(5,0)--cycle;
        \filldraw[fill=black!10]
        (1,0)--(1,1)--(2,1)--(2,2)--(3,2)--(3,3)--(4,3)--(4,4)--(5,4)--(5,5)--(6,5)--(6,-4)--(5,-4)--(5,0)--cycle;

        \draw[->, thick, black]  (0,0)--(0,7);
        \draw[->, thick, black]  (0,0)--(7,0);

        \draw (1,0)--(1,6); \draw (2,0)--(2,6); \draw
        (3,0)--(3,6); \draw (4,0)--(4,6); \draw (5,0)--(5,6);

        \draw (0,1)--(6,1); \draw (0,2)--(6,2); \draw
        (0,3)--(6,3); \draw (0,4)--(6,4); \draw (0,5)--(6,5);
        \draw (5,-1)--(6,-1); \draw (5,-2)--(6,-2); \draw (5,-3)--(6,-3);

        \draw (3,7) node{$\Ea{Y}^1_{p,q}$};
        \draw (0.5,-0.5) node{$0$}; \draw (-0.5,0.5) node{$0$}; \draw (7,-0.5) node{$p$}; \draw (-0.5,7) node{$q$};


        \draw[->,thick] (7,3)--(9,3);
        \draw (8,4) node{$\fa_*^1$};


        \filldraw[fill=black!40]
        (10,0)--(10,1)--(11,1)--(11,2)--(12,2)--(12,3)--(13,3)--(13,4)--(14,4)--(14,5)--(15,5)--(15,6)--(16,6)--(16,-4)--(15,-4)--(15,0)--cycle;
        \filldraw[fill=black!10]
        (11,0)--(11,1)--(12,1)--(12,2)--(13,2)--(13,3)--(14,3)--(14,4)--(15,4)--(15,6)--(16,6)--(16,-4)--(15,-4)--(15,0)--cycle;

        \draw[->, thick, black]  (10,0)--(10,7);
        \draw[->, thick, black]  (10,0)--(17,0);

        \draw (11,0)--(11,2); \draw (12,0)--(12,3); \draw
        (13,0)--(13,4); \draw (14,0)--(14,5); \draw (15,0)--(15,6);

        \draw (10,1)--(16,1); \draw (11,2)--(16,2); \draw
        (12,3)--(16,3); \draw (13,4)--(16,4); \draw (14,5)--(16,5);
        \draw (15,-1)--(16,-1); \draw (15,-2)--(16,-2); \draw (15,-3)--(16,-3);

        \draw (13,7) node{$\Ea{X}^1_{p,q}$};
        \draw (11.5,4.5) node{$0$};
    \end{tikzpicture}
\end{center}
\caption{The induced map of spectral sequences is an isomorphism
below the diagonal and injective on the diagonal.}
\label{figSpecSeqGeneral}
\end{figure}
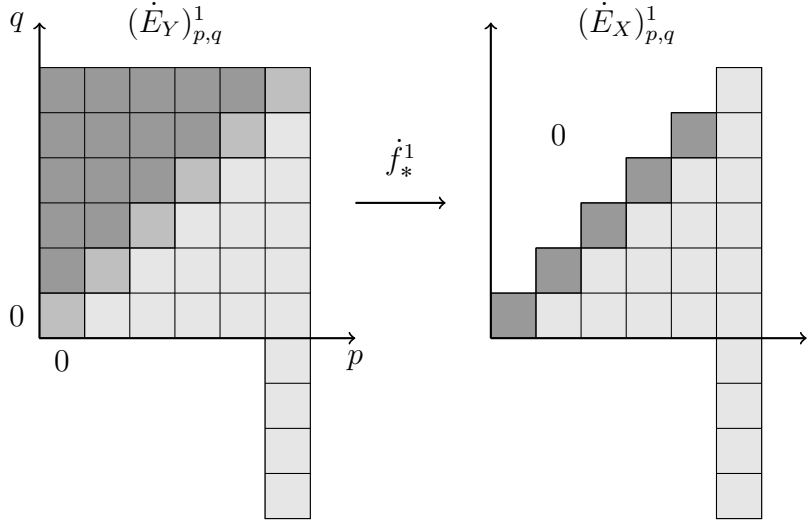

\begin{prop}
The map $\fa_*^1\colon \Ea{Y}^1_{p,q} \to \Ea{X}^1_{p,q}$ is an
isomorphism when $p>q$ or $p=q=n$. It is injective when $p=q$. In
other cases, i.e. when $p<q$, the modules $\Ea{X}^1_{p,q}$ vanish.
\end{prop}

\begin{proof}
The map $f_*^2\colon \E{\dd Y}^2_{p,q}\to \E{\dd X}^2_{p,q}$ is an
isomorphism for $p>q$ and injective for $p=q$ (see \cite[Th.3 and
Remark 6.6]{AyV1}). Thus $\fa_*^1\colon \Ea{Y}^1_{p,q}\to
\Ea{X}^1_{p,q}$ is an isomorphism for $q<p<n$ and injective for
$q=p<n$. Note that in \cite{AyV1} we considered manifolds with
corners, but the argument used there can be applied to any
Buchsbaum pseudo-cell complex without significant changes.

As for the case $p=n$, the map $f_*\colon \E{Y}^1_{n,q}\to
\E{X}^1_{n,q}$ is an isomorphism since the identification $\sim$
does not touch the interior of $Y$ and, therefore,
\[
X_n/X_{n-1} = X/\dd X \cong Y/\dd Y = Y_n/Y_{n-1}.
\]
Thus $\fa_*^1\colon \Ea{Y}^1_{n,q}=H_{n+q}(Y,\dd Y)\to
\Ea{X}^1_{n,q} =H_{n+q}(X,\dd X)$ is an isomorphism by excision.
\end{proof}

This proposition together with \eqref{eqKunnethYSpecSec} and
Proposition \ref{propSpecSeqQForm} gives a complete description of
differentials and non-diagonal terms of $\Ea{X}^r_{*,*}$.

\begin{thm}\label{thmEplus1struct}
Let $Q$ be a Buchsbaum (over $\ko$) pseudo-cell complex, and let
$X=(Q\times T^n)/\simc$ be the quotient construction determined by
some $\ko$-characteristic function on $Q$. There exists a
homological spectral sequence $\Ea{X}^r_{*,*}$
converging to $H_*(X)$. From its second page this spectral
sequence coincides with $\E{X}^*_{*,*}$, the spectral sequence
associated with the orbit type filtration. The first page,
$\Ea{X}^1$ is the homology module of $\E{X}^1$ with respect to the
differential $\difm{X}$ of degree $(-1,0)$. The following
properties hold for $\Ea{X}^*_{*,*}$:

\begin{enumerate}
\item Non-diagonal terms of the first page have the form
\[
\Ea{X}^1_{p,q}\cong\begin{cases} H_p(\dd Q)\otimes\Lambda_q,\mbox{
if }q<p<n;\\
\bigoplus\limits_{q_1+q_2=q+n}H_{q_1}(Q,\dd
Q)\otimes\Lambda_{q_2},\mbox{ if }p=n;\\
0, \mbox{ if }q>p;
\end{cases}
\]

\item
There exist injective maps $\fa_*^1\colon H_q(\dd
Q)\otimes\Lambda_{q}\hookrightarrow \Ea{X}^1_{q,q}$ for
$q\leqslant n$.

\item Nontrivial differentials for $r\geqslant 1$ have the form
\[
\difa{X}^r\cong\begin{cases}
\delta_{q_1}\otimes\id_{\Lambda}\colon
\overset{\substack{\Ea{X}^1_{n,q_1+q_2-n}\\
\cup}}{H_{q_1}(Q,\dd Q)\otimes\Lambda_{q_2}}\to \overset{\substack{\Ea{X}^1_{q_1-1,q_2}\\
\cup}}{H_{q_1-1}(\dd Q)\otimes\Lambda_{q_2}}, \\
\hfill \mbox{
if } r=n-q_1+1, q_1-1>q_2;\\
\fa_*^1\circ(\delta_{q_1}\otimes \id_{\Lambda})\colon
H_{q_1}(Q,\dd Q)\otimes\Lambda_{q_2}\to H_{q_1-1}(\dd
Q)\otimes\Lambda_{q_2}\hookrightarrow
\Ea{X}^*_{q_1-1,q_1-1},\qquad
\\ \hfill \mbox{ if } r=n-q_1+1, q_1-1=q_2;\\
0, \mbox{ otherwise}.
\end{cases}
\]
\end{enumerate}
\end{thm}

\subsection{Diagonal terms of the spectral sequence}

Our next goal is to compute the diagonal terms $\Ea{X}^1_{q,q}$,
since they are not described explicitly by Theorem
\ref{thmEplus1struct}. In this subsection we state the results
about the dimensions of these modules. The proofs are given in the
next section.

Let $\br_p(S)=\dim\Hr_p(S)$ for $p<n$. If $Q$ is a Buchsbaum
pseudo-cell complex, we have $\dim\Hr_p(\dd Q) = \br_p(S_Q)$,
since $S_Q$ is homologous to $\dd Q$ by Proposition
\ref{propUnivers}. Let $h_q(S)$, $h'_q(S)$, and $h''_q(S)$ be the
$h$-, $h'$-, and $h''$-numbers of a simplicial poset $S$ (see
definitions in Section \ref{secHvectors}).

\begin{thm}\label{thmBorderStruct}
In the notation and under conditions of Theorem
\ref{thmEplus1struct} there holds
\[
\dim \Ea{X}^1_{q,q} = h_q(S_Q)+{n\choose
q}\sum\limits_{p=0}^{q}(-1)^{p+q}\br_p(S_Q)
\]
for $q\leqslant n-1$.
\end{thm}

\begin{thm}\label{thmBorderManif}
Let $Q$ be a Buchsbaum manifold with corners and $X=(Q\times
T^n)/\simc$. Then:

(1) $\dim \Ea{X}^1_{q,q}=h'_{n-q}(S_Q)$ for $q\leqslant n-2$, and
$\dim\Ea{X}^1_{n-1,n-1}=h'_1(S_Q)+n$.

(2) $\dim \E{X}^2_{q,q} = \dim\Ea{X}^2_{q,q} = h'_{n-q}(S_Q)$ for
$0\leqslant q\leqslant n$.
\end{thm}

For the cone over Buchsbaum simplicial poset, the diagonal
components of $\infty$-page also have a clear combinatorial
meaning.

\begin{thm}\label{thmHtwoPrimes}
Let $S$ be a Buchsbaum simplicial poset, $P=P(S)$ be the cone over
its geometric realization, and $X=(P\times T^n)/\simc$. Then
\[
\dim\E{X}^{\infty}_{q,q}=\dim\Ea{X}^{\infty}_{q,q}=h''_q(S)
\]
for $0\leqslant q\leqslant n$.
\end{thm}

\begin{cor}
If $S$ is Buchsbaum, then $h''_i(S)\geqslant 0$.
\end{cor}

\begin{proof}
For any simplicial poset $S$ there exists a characteristic
function on $P=P(S)$ over rational numbers. Thus we can consider
the space $X=(P\times T^n)/\simc$ and apply Theorem
\ref{thmHtwoPrimes}.
\end{proof}

%
%
%
%
%
%
%

\section{Face vectors and ranks of diagonal components}\label{secHvectors}

In this section we prove Theorems \ref{thmBorderStruct},
\ref{thmBorderManif} and \ref{thmHtwoPrimes}.

\subsection{Preliminaries on face vectors}
First recall several standard definitions from combinatorial
theory of simplicial posets.

\begin{con}\label{conHvectors}
Let $S$ be a pure simplicial poset, $\dim S=n-1$. Let $f_i(S)$ be
the the number of $i$-dimensional simplices in $S$ and, in
particular, $f_{-1}(S)=1$ (the element $\minel\in S$ has dimension
$-1$). The array $(f_{-1},f_0,\ldots,f_{n-1})$ is called the
$f$-vector of~$S$. We write $f_i$ instead of $f_i(S)$ because the
poset is always clear from the context. Let $f_S(t)$ be the
generating polynomial: $f_S(t) = \sum_{i\geqslant 0}f_{i-1}t^i$.

Define $h$-numbers by the relation:
\begin{equation}\label{eqHvecDefin}
\sum_{i=0}^nh_it^i=\sum_{i=0}^nf_{i-1}t^i(1-t)^{n-i}=
(1-t)^nf_S\left(\dfrac{t}{1-t}\right).
\end{equation}
Let $\be_i(S) = \dim H_i(S)$, $\br_i(S)=\dim \Hr_i(S)$, and
\[
\chi(S)=\sum_{i=0}^{n-1}(-1)^i\be_i(S)=
\sum_{i=0}^{n-1}(-1)^if_i(S)
\]
\[
\chir(S)=\sum_{i=0}^{n-1}\br_i(S)=\chi(S)-1.
\]
Thus $f_S(-1)=1-\chi(S)$. Note that
\begin{equation}\label{eqHnEulerChar}
h_{n}=(-1)^{n-1}\chir(S).
\end{equation}

Define $h'$- and $h''$-numbers of $S$ by the formulas
\[
h_i'=h_i+{n\choose
i}\left(\sum_{j=1}^{i-1}(-1)^{i-j-1}\br_{j-1}(S)\right)\mbox{ for
} 0\leqslant i\leqslant n;
\]
\[
h_i'' = h_i'-{n\choose i}\br_{i-1}(S) = h_i+{n\choose
i}\left(\sum_{j=1}^{i}(-1)^{i-j-1}\br_{j-1}(S)\right)\mbox{ for }
0\leqslant i\leqslant n-1,
\]
and $h''_n=h'_n$. The summation over an empty set is assumed to be
zero. From \eqref{eqHnEulerChar} there follows
\begin{equation}\label{eqHnprime}
h'_n=h_n+\sum_{j=0}^{n-1}(-1)^{n-j-1}\br_{j-1}(S)= \br_{n-1}(S).
\end{equation}
\end{con}

\begin{stm}[Dehn--Sommerville relations]
For a homology manifold $S$ there holds
\begin{equation}\label{eqDehnSomQuasiManif}
h_i=h_{n-i}+(-1)^i{n\choose i}(1-(-1)^n-\chi(S)),
\end{equation}
or, equivalently:
\begin{equation}\label{eqDehnSomQuasiManifChir}
h_i=h_{n-i}+(-1)^i{n\choose i}(1+(-1)^n\chir(S)).
\end{equation}
Moreover, $h''_i=h''_{n-i}$.
\end{stm}

\begin{proof}
The first statement can be found in e.g. \cite{St} or
\cite[Thm.3.8.2]{BPnew}. Also see Remark \ref{remDehnSomManif}
below. The last statement follows from the definition of
$h''$-vector and Poincare duality $\be_i(S)=\be_{n-1-i}(S)$ (see
\cite[Lm.7.3]{No}).
\end{proof}

Now we introduce an auxiliary numerical characteristic of a
simplicial poset $S$.

\begin{defin}
Let $S$ be a Buchsbaum simplicial poset. For $i\geqslant 0$
consider the number
\[
\ft_i(S)=\sum_{I\in S,\dim I=i}\dim \Hr_{n-1-|I|}(\lk_SI).
\]
\end{defin}

For a homology manifold $S$ there holds $\ft_i=f_i$ since all
proper links are homology spheres. In general, there is another
formula connecting these quantities.

\begin{prop}\label{propFtildeToF}
For Buchsbaum simplicial poset $S$ there holds
\[
f_S(t)=(1-\chi(S))+(-1)^n\sum_{k\geqslant
0}\ft_{k}(S)\cdot(-t-1)^{k+1}.
\]
\end{prop}

\begin{proof}
This follows from the general statement
\cite[Th.9.1]{MMP},\cite[Th.3.8.1]{BPnew}, but we provide an
independent proof for completeness. As stated in
\cite[Lm.3.7,3.8]{AyBuch} for simplicial complexes (and not
difficult to prove for simplicial posets):
\[
\frac{d}{dt}f_S(t)=\sum_{v\in\ver(S)}f_{\lk v}(t),
\]
and, more generally,
\[
\left(\frac{d}{dt}\right)^kf_S(t)=k!\sum_{I\in S, |I|=k}f_{\lk
I}(t).
\]
Thus for $k\geqslant 1$:
\begin{equation*}
\begin{split}
f_S^{(k)}(-1) &=k!\sum_{I\in S, |I|=k}(1-\chi(\lk_SI))= \\&=
k!\sum_{I\in S, |I|=k}(-1)^{n-|I|}\dim \Hr_{n-|I|-1}(\lk I)=
(-1)^{n-k}k!\ft_{k-1}.
\end{split}
\end{equation*}
The Taylor expansion of $f_S(t)$ at the point $-1$ has the form:
\[
f_S(t)=f_S(-1)+\sum_{k\geqslant
1}\frac{1}{k!}f_S^{(k)}(-1)(t+1)^k= (1-\chi(S))+\sum_{k\geqslant
0}(-1)^{n-k-1}\ft_k\cdot(t+1)^{k+1}.
\]
This finishes the proof.
\end{proof}

\begin{rem}\label{remDehnSomManif}
If $S$ is a homology manifold, then Proposition
\ref{propFtildeToF} implies
\[
f_S(t) = (1-(-1)^n-\chi(S))+(-1)^nf_S(-t-1),
\]
which is yet another equivalent form of Dehn--Sommerville
relations \eqref{eqDehnSomQuasiManif}.
\end{rem}

\begin{lemma}\label{lemmaHfromFt}
For Buchsbaum poset $S$ there holds
\[
\sum_{i=0}^nh_it^i = (1-t)^n(1-\chi(S)) + \sum_{k\geqslant
0}\ft_k\cdot(t-1)^{n-k-1}.
\]
\end{lemma}

\begin{proof}
Substitute $t/(1-t)$ in Proposition \ref{propFtildeToF} and apply
\eqref{eqHvecDefin}.
\end{proof}

Comparing the coefficients at $t^i$ in the identity of Lemma
\ref{lemmaHfromFt} we get:
\begin{equation}\label{eqFtildeToH}
h_i(S)=(1-\chi(S))(-1)^i{n\choose i}+\sum_{k\geqslant
0}(-1)^{n-k-i-1}{n-k-1\choose i} \ft_k(S).
\end{equation}

\subsection{Ranks of $\E{X}^1_{*,*}$}

To prove Theorem \ref{thmBorderStruct} we use the following
straightforward idea. The module $\Ea{X}^1$ is the homology of
$\E{X}^1$ with respect to the differential $\difm{X}$ of degree
$(-1,0)$. Theorem \ref{thmEplus1struct} describes the ranks of all
groups $\Ea{X}^1_{p,q}$ except for $p=q$; the terms
$\E{X}^1_{p,q}$ are known as well. Thus the ranks of the remaining
terms $\dim \Ea{X}^1_{q,q}$ can be found from the equality of
Euler characteristics, computed for $\E{X}^1$ and $\Ea{X}^1$. When
we pass from $\E{X}^1$ to $\Ea{X}^1$, the terms with $p=n$ do not
change; the other groups are the same as if we passed from $\E{\dd
X}^1$ to $\E{\dd X}^2$. Thus it is sufficient to perform
calculations with the truncated sequence $\E{\dd X}^*$.

Let $\chi^1_q$ be the Euler characteristic of the $q$-th row of
$\E{\dd X}^1_{*,*}$:
\begin{equation}\label{eqChi1def}
\chi_q^1=\sum_{p\leqslant n-1}(-1)^p\dim\E{\dd X}^1_{p,q}.
\end{equation}

\begin{lemma}\label{lemmaChic1}
For $q\leqslant n-1$ we have $\chi_q^1=(\chi(S_Q)-1){n\choose
q}+(-1)^qh_q(S_Q)$.
\end{lemma}

\begin{proof}
By Proposition \ref{propUnivers} there is an isomorphism of
spectral sequences $\E{\dd Q}^*\to \E{\dd P(S_Q)}^*$. Thus, in
particular, for any $I\in S_Q\setminus\{\minel\}$, $|I|=n-p$ we
have an isomorphism
\begin{equation}\label{eqUnivLk}
H_p(F_I,\dd F_I)\cong H_p(G_I,\dd G_I)\cong H_{p-1}(\lk_{S_Q}I)
\end{equation}
where $G_I$ is the face of $P(S_Q)$ dual to $I$. The last
isomorphism in \eqref{eqUnivLk} is due to the long exact sequence
of the pair $(G_I,\dd G_I)$, since $G_I=\cone (\dd G_I)$ and $\dd
G_I\cong \lk_{S_Q}I$.

For $p<n$ we have
\begin{multline*}
\dim \E{\dd X}^1_{p,q} = \sum_{I, \dim F_I=p}\dim H_{p+q}(X_I,\dd X_I)=\\
= \sum_{I, |I|=n-p}\dim (H_p(F_I,\dd F_I)\otimes H_q(T^n/T_I)) =
{p\choose q}\cdot \ft_{n-p-1}(S_Q).
\end{multline*}
In the last equality we used \eqref{eqUnivLk} and the definition
of $\ft$-numbers. Therefore,
\begin{equation}\label{eqChi1eq}
\chi_q^1=\sum_{p\leqslant n-1}(-1)^p\dim\E{\dd
X}^1_{p,q}=\sum_{p\leqslant n-1}(-1)^p{p\choose
q}\ft_{n-p-1}(S_Q).
\end{equation}

Now substitute $i=q$ and $k=n-p-1$ in \eqref{eqFtildeToH} and
combine it with \eqref{eqChi1eq}.
\end{proof}

\subsection{Ranks of $\Ea{X}^1_{*,*}$}

By construction of the modified spectral sequence,
$\Ea{X}^1_{p,q}\cong\E{\dd X}^2_{p,q}$ for $p\leqslant n-1$. Let
$\chi^2_q$ be the Euler characteristic of $q$-th row of $\E{\dd
X}^2_{*,*}$:
\begin{equation}\label{eqChiE2def}
\chi^2_q = \sum_{p\leqslant n-1}(-1)^p\dim\E{\dd X}^2_{p,q}.
\end{equation}
Euler characteristics of the first and the second pages coincide:
$\chi_q^2=\chi_q^1$. By Theorem \ref{thmEplus1struct}, for $q<p<n$
we have
\[
\dim\Ea{X}^1_{p,q}={n\choose q}\be_p(S_Q).
\]
Lemma \ref{lemmaChic1} yields
\[
(-1)^q\dim\Ea{X}^1_{q,q}+\sum_{p=q+1}^{n-1}(-1)^p{n\choose q}
\be_p(S_Q)=(\chi(S_Q)-1){n\choose q}+(-1)^qh_q(S_Q).
\]
By taking into account the equality
$\chi(S_Q)=\sum_{p=0}^{n-1}\be_p(S_Q)$ and the obvious relation
between reduced and non-reduced Betti numbers, this proves
Theorem~\ref{thmBorderStruct}.

\subsection{Manifold case} Now we prove Theorem
\ref{thmBorderManif}. If $Q$ is a Buchsbaum manifold with corners,
then $S_Q$ is a homology manifold. Then Poincare duality
$\be_i(S_Q)=\be_{n-1-i}(S_Q)$ and Dehn--Sommerville relations
\eqref{eqDehnSomQuasiManifChir} imply
\begin{equation*}
\begin{split}
\dim\Ea{X}^1_{q,q}&=h_q+{n\choose q}\sum_{p=0}^q(-1)^{p+q}\br_p=\\
&=h_q-(-1)^q{n\choose q}+{n\choose
q}\sum_{p=0}^q(-1)^{p+q}\be_p=\\
&=h_q-(-1)^q{n\choose q}+{n\choose
q}\sum_{p=n-1-q}^{n-1}(-1)^{n-1-p+q}\be_p=\\
&=h_{n-q}+(-1)^q{n\choose
q}\left[-(-1)^n+(-1)^n\chi+\sum_{p=n-1-q}^{n-1}(-1)^{n-1-p}\be_p\right]=\\
&=h_{n-q}+(-1)^q{n\choose q}
\left[-(-1)^n+\sum_{p=0}^{n-q-2}(-1)^{p+n}\be_p \right].
\end{split}
\end{equation*}

The last expression in brackets coincides with
$\sum_{p=-1}^{n-q-2}(-1)^{p+n}\br_p$ whenever the summation is
taken over nonempty set, that is for $q\leqslant n-2$. Thus
$\dim\Ea{X}^1_{q,q}=h'_{n-q}$ for $q\leqslant n-2$. In the case
$q=n-1$ we have $\dim\Ea{X}^1_{n-1,n-1}=h_1+{n\choose
n-1}=h'_1+n$. This proves part (1) of Theorem
\ref{thmBorderManif}.

Part (2) follows easily. Indeed, for $q=n$ we have
\[
\dim\Ea{X}^2_{n,n}=\dim\Ea{X}^1_{n,n}={n\choose n}\dim H_n(Q,\dd
Q)=1=h'_0(S_Q)
\]
For $q=n-1$:
\[
\dim \Ea{X}^2_{n-1,n-1}=\dim\Ea{X}^1_{n-1,n-1}-{n\choose
n-1}\dim\im \delta_n=h'_1(S_Q),
\]
since the map $\delta_n\colon H_n(Q,\dd Q) \to H_{n-1}(\dd Q)$ is
injective and $\dim H_n(Q,\dd Q)=1$.

If $q\leqslant n-2$, then $\Ea{X}^2_{q,q}=\Ea{X}^1_{q,q}$, and the
statement follows from part (1).

\subsection{Cone case}

If $P = P(S) \cong \cone|S|$, then the map $\delta_i\colon
H_i(P,\dd P)\to \Hr_{i-1}(\dd P)$ is an isomorphism as follows
from the long exact sequence of the pair $(P,\dd P)$. Thus for
$q\leqslant n-1$, Theorem \ref{thmEplus1struct} implies
\[
\dim \Ea{X}^{\infty}_{q,q}= \dim \Ea{X}^1_{q,q}-{n\choose q}\dim
H_{q+1}(P,\dd P) = \dim \Ea{X}^1_{q,q}-{n\choose q}\br_q(S).
\]
By Theorem \ref{thmBorderStruct} this expression is equal to
\[
h_q+{n\choose q}
\left[\sum_{p=0}^q(-1)^{p+q}\br_p(S)\right]-{n\choose q}\br_q(S) =
h_q(S)+{n\choose q}\sum_{p=0}^{q-1}(-1)^{p+q}\br_p(S)=h''_{q}.
\]
The case $q=n$ follows from \eqref{eqHnprime}. Indeed, the term
$\Ea{X}^1_{n,n}$ survives in the spectral sequence, thus:
\[
\dim \Ea{X}^{\infty}_{n,n} = {n\choose n}\dim H_n(P,\dd
P)=\be_{n-1}(S) = h_{n}'(S)=h_n''(S).
\]
This proves Theorem \ref{thmHtwoPrimes}.

%
%
%
%
%
%
%

\section{Homology of $X$.}\label{secHomology}

In this section we suppose $\ko$ is a field. Theorem
\ref{thmEplus1struct} gives an additional grading on $H_*(X)$,
namely the one induced by degrees of exterior forms, as described
below. In the following $Q$ is an arbitrary Buchsbaum pseudo-cell
complex of dimension $n$.

\begin{figure}[h]
\begin{center}
    \begin{tikzpicture}[scale=.6]


        \filldraw[fill=black!40!green!95]
        (0,0)--(0,1)--(6,1)--(6,-4)--(5,-4)--(5,0)--cycle;
        \filldraw[fill=black!40!green!70]
        (0,1)--(0,2)--(6.05,2)--(6.05,-3)--(5.05,-3)--(5.05,1)--cycle;
        \filldraw[fill=black!40!green!50]
        (0,2)--(0,3)--(6.1,3)--(6.1,-2)--(5.1,-2)--(5.1,2)--cycle;
        \filldraw[fill=black!40!green!35]
        (0,3)--(0,4)--(6.15,4)--(6.15,-1)--(5.15,-1)--(5.15,3)--cycle;
        \filldraw[fill=black!40!green!20]
        (0,4)--(0,5)--(6.2,5)--(6.2,0)--(5.2,0)--(5.2,4)--cycle;
        \filldraw[fill=black!40!green!5]
        (0,5)--(0,6)--(6.25,6)--(6.25,1)--(5.25,1)--(5.25,5)--cycle;

        \draw[->, thick, black]  (0,0)--(0,7);
        \draw[->, thick, black]  (0,0)--(7,0);

        \draw (1,0)--(1,6); \draw (2,0)--(2,6); \draw
        (3,0)--(3,6); \draw (4,0)--(4,6); \draw (5.25,5)--(5.25,6);

        \draw (0,1)--(6.25,1); \draw (0,2)--(6.25,2); \draw
        (0,3)--(6.25,3); \draw (0,4)--(6.25,4); \draw (0,5)--(6.25,5);

        \draw (5,-1)--(6,-1); \draw (5,-2)--(6,-2); \draw (5,-3)--(6,-3);

        \draw (3,7) node{$\Ea{Y}^1_{*,*}$};

        \draw (-1,0.5) node{\tiny $\Ead{Y}{0}^1_{*,*}$}; \draw (-1,1.5) node{\tiny
        $\Ead{Y}{1}^1_{*,*}$}; \draw (-1,3.5) node{$\vdots$};
        \draw (-1,5.5) node{\tiny $\Ead{Y}{n}^1_{*,*}$};


        \draw[->,thick] (7,3)--(9,3);
        \draw (8,4) node{$\fa_*^1$};

        \filldraw[fill=black!40!yellow!95]
        (10,0)--(10,1)--(16,1)--(16,-4)--(15,-4)--(15,0)--cycle;
        \filldraw[fill=black!40!yellow!70]
        (11,1)--(11,2)--(16.05,2)--(16.05,-3)--(15.05,-3)--(15.05,1)--cycle;
        \filldraw[fill=black!40!yellow!50]
        (12,2)--(12,3)--(16.1,3)--(16.1,-2)--(15.1,-2)--(15.1,2)--cycle;
        \filldraw[fill=black!40!yellow!35]
        (13,3)--(13,4)--(16.15,4)--(16.15,-1)--(15.15,-1)--(15.15,3)--cycle;
        \filldraw[fill=black!40!yellow!20]
        (14,4)--(14,5)--(16.2,5)--(16.2,0)--(15.2,0)--(15.2,4)--cycle;
        \filldraw[fill=black!40!yellow!5]
        (16.25,6)--(16.25,1)--(15.25,1)--(15.25,6)--cycle;

        \draw[->, thick, black]  (10,0)--(10,7);
        \draw[->, thick, black]  (10,0)--(17,0);

        \draw (11,0)--(11,2); \draw (12,0)--(12,3); \draw
        (13,0)--(13,4); \draw (14,0)--(14,5);

        \draw (10,1)--(16.25,1); \draw (11,2)--(16.25,2); \draw
        (12,3)--(16.25,3); \draw (13,4)--(16.25,4); \draw (14,5)--(16.25,5);

        \draw (15,-1)--(16,-1); \draw (15,-2)--(16,-2); \draw (15,-3)--(16,-3);

        \draw (13,7) node{$\Ea{X}^1_{p,q}$};
        \draw (9,0.5) node{\tiny $\Ead{X}{0}^1_{*,*}$}; \draw (10,1.5) node{\tiny
        $\Ead{X}{1}^1_{*,*}$}; \draw (11,2.5) node{\tiny $\Ead{X}{2}^1_{*,*}$};
        \draw (14.2,5.5) node{\tiny $\Ead{X}{n}^1_{*,*}$};

    \end{tikzpicture}
\end{center}
\caption{Decomposition of spectral sequences into graded
components} \label{figSpecSeqLayers}
\end{figure}
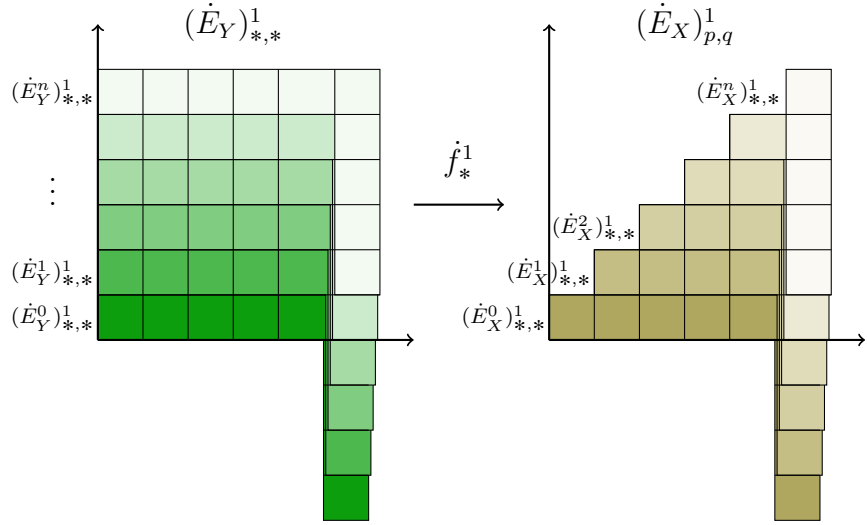

\begin{con}
The spectral sequence $\Ea{Y}^*$ splits in the direct sum of
spectral subsequences, indexed by degrees of exterior forms. For
$0\leqslant j\leqslant n$ consider the \ang-shaped spectral
sequence
\[
\Ead{Y}{j}^r_{p,q}=\Ea{Q}^r_{p,q-j}\otimes \Lambda_j.
\]
Clearly, $\Ea{Y}^r_{*,*}=\bigoplus_{j=0}^n\Ead{Y}{j}^r_{*,*}$.
This decomposition is sketched on Figure \ref{figSpecSeqLayers}.
Let $H_{i,j}(Y)$ denote the module $H_i(Q)\otimes\Lambda_j$. Then
$\Ead{Y}{j}^r_{p,q}\Rightarrow H_{p+q-j,j}(Y)$.

Let us construct the corresponding \ang-shaped spectral
subsequences in $\Ea{X}^*_{*,*}$. Consider the bigraded vector
subspaces $\Ead{X}{j}^1_{*,*}$:
\[
\Ead{X}{j}^1_{p,q}=\begin{cases} \Ea{X}^1_{p,q}, \mbox{ if } q=j
\mbox{ and } p<n;\\
0, \mbox{ if } q\neq j \mbox{ and } p<n;\\
H_{q+n-j}(Q,\dd Q)\otimes \Lambda_j, \mbox{ if } p=n.
\end{cases}
\]
In the last case we used the isomorphism of Theorem
\ref{thmEplus1struct}. Theorem \ref{thmEplus1struct} implies that
all differentials of $\Ea{X}^*_{*,*}$ preserve the subspace
$\Ead{X}{j}^*_{*,*}$, thus spectral subsequences
$\Ead{X}{j}^r_{*,*}$ are well defined for $r\geqslant 2$, and
$\Ea{X}^r_{*,*}=\bigoplus_{j=0}^n\Ead{X}{j}^r_{*,*}$.

Over a field, $H_k(X)$ can be identified with the associated
module $\bigoplus_{p+q=k} \E{X}^{\infty}_{p,q}$, and thus inherits
a double grading:
\[
H_k(X)\cong \bigoplus_{i+j=k} H_{i,j}(X),
\]
where
\[
H_{i,j}(X)\eqd\bigoplus_{p+q=i+j} \Ead{X}{j}^{\infty}_{p,q}.
\]
Thus we have $\Ead{X}{j}^r_{p,q}\Rightarrow H_{p+q-j,j}(X)$. The
map $\fa_*^r\colon \Ea{Y}^r\to\Ea{X}^r$ sends $\Ead{Y}{j}^r$ to
$\Ead{X}{j}^r$ for each $j\in\{0,\ldots,n\}$. The map $f_*\colon
H_*(Y)\to H_*(X)$ sends $H_{i,j}(Y)$ to $H_{i,j}(X)$.
\end{con}

\begin{thm}\label{thmHomolX}\mbox{}

\begin{enumerate}
\item If $i>j$, then $f_*\colon H_{i,j}(Y)\to
H_{i,j}(X)$ is an isomorphism. As a consequence, $H_{i,j}(X)\cong
H_i(Q)\otimes\Lambda_j$.

\item If $i<j$, then there exists an isomorphism $H_{i,j}(X)\cong H_i(Q,\dd
Q)\otimes \Lambda_j$.

\item In case $i=j<n$, the module $H_{i,i}(X)$ fits in the exact sequence
\[
0\rightarrow\Ea{X}^{\infty}_{i,i}\rightarrow H_{i,i}(X)\rightarrow
H_i(Q,\dd Q)\otimes\Lambda_i \rightarrow 0,
\]
or, equivalently,
\[
0\rightarrow \im\delta_{i+1}\otimes \Lambda_i\rightarrow
\Ea{X}^1_{i,i}\rightarrow H_{i,i}(X)\rightarrow H_i(Q,\dd
Q)\otimes\Lambda_i\rightarrow 0
\]

\item If $i=j=n$, then
\[
H_{n,n}(X)=\Ea{X}^{\infty}_{n,n}=\Ea{X}^1_{n,n}\cong H_n(Q,\dd Q).
\]
\end{enumerate}
\end{thm}

\begin{proof}
According to Theorem \ref{thmEplus1struct}, the map
$\fa_*^1\colon\Ead{Y}{j}^1_{i,q}\to\Ead{X}{j}^1_{i,q}$ is an
isomorphism if $i>j$ or $i=j=n$, and injective if $i=j$. For each
$j$ both spectral sequences $\Ead{Y}{j}$ and $\Ead{X}{j}$ are
\ang-shaped, thus fold in the long exact sequences:
\begin{equation}\label{eqTwoLongGrad}
\xymatrix{\ldots\ar@{->}[r]&\Ead{Y}{j}^1_{i,j}\ar@{->}[r]\ar@{->}[d]^{\fa_*^1}&
H_{i,j}(Y)\ar@{->}[r]\ar@{->}[d]^{f_*}&\Ead{Y}{j}^1_{n,i-n+j}
\ar@{->}[r]^{\difa{Y}^{n-i+1}}\ar@{->}[d]^{\fa_*^1}_{\cong}&
\Ead{Y}{j}^1_{i-1,j}\ar@{->}[r]\ar@{->}[d]^{f_*}&\ldots\\
\ldots\ar@{->}[r]&\Ead{X}{j}^1_{i,j}\ar@{->}[r]&
H_{i,j}(X)\ar@{->}[r]&\Ead{X}{j}^1_{n,i-n+j}\ar@{->}[r]^{\difa{X}^{n-i+1}}&
\Ead{X}{j}^1_{i-1,j}\ar@{->}[r]&\ldots}
\end{equation}
Application of five lemma in the case $i>j$ proves (1). For $i<j$,
the groups $\Ead{X}{j}^1_{i,j}$, $\Ead{X}{j}^1_{i-1,j}$ vanish by
dimensional reasons, thus $H_{i,j}(X) \cong \Ead{X}{j}^1_{n,i-n+j}
\cong \Ead{Y}{j}^1_{n,i-n+j} \cong H_i(Q,\dd Q)\otimes\Lambda_j$.
Case $i=j$ also follows from \eqref{eqTwoLongGrad} by a simple
diagram chase.
\end{proof}

In case of manifolds Theorem \ref{thmHomolX} reveals a bigraded
duality. If $Q$ is a nice manifold with corners, $Y=Q\times T^n$,
and $\lambda$ is a characteristic function over $\Zo$, then
$X=Y/\simc$ is a compact orientable topological manifold with the
locally standard torus action. In this case Poincare duality
respects the double grading.

\begin{prop}\label{propXmanifoldDuality}
Let $Q$ be a Buchsbaum manifold with corners and $X$ is a quotient
construction over $Q$ determined by a $\Zo$-characteristic
function. Then $H_{i,j}(X;\ko)\cong H_{n-i,n-j}(X;\ko)$ for any
field $\ko$.
\end{prop}

\begin{proof}
When $i<j$, we have
\[
H_{i,j}(X)\cong H_i(Q,\dd Q)\otimes\Lambda_j\cong
H_{n-i}(Q)\otimes\Lambda_{n-j}\cong H_{n-i,n-j}(X),
\]
by the Poincare--Lefschetz duality applied to $Q$ and Poincare
duality applied to torus. The remaining isomorphism
$H_{i,i}(X)\cong H_{n-i,n-i}(X)$ now follows from the ordinary
Poincare duality in~$X$.
\end{proof}

\begin{rem}\label{remOrbifoldDuality}
If $X$ is determined by $\Qo$-characteristic function, then it is
a homology $\Qo$-manifold. In this case Proposition
\ref{propXmanifoldDuality} holds over $\Qo$.
\end{rem}

%
%
%
%
%
%
%

\section{One example with non-acyclic proper faces}\label{secExample}

Let $Q$ be the product of $S^1$ with the closed interval
$\mathbb{I}=[-1,1]\subset \Ro^1$. Then $Q$ is a nice manifold with
corners having two proper faces: $F_1=S^1\times \{-1\}$ and
$F_2=S^1\times \{1\}$. The faces are not acyclic, so the arguments
of the paper cannot be applied. Consider the $2$-torus $T^2$ with
a given coordinate splitting $T^2=T^{(\{1\})}\times T^{(\{2\})}$.

First, define the characteristic function $\lambda$ on $Q$ by
\[
\lambda(F_1)=T^{(\{1\})},\qquad \lambda(F_2)=T^{(\{2\})}.
\]
The corresponding manifold with locally standard action is
\[
X=(S^1\times\mathbb{I}\times T^2)/\simc=S^1\times(\mathbb{I}\times
T^2/\simc)=S^1\times \mathcal{Z}_{\mathbb{I}}\cong S^1\times S^3.
\]
Here $\mathcal{Z}_{\mathbb{I}}$ is the moment-angle manifold of
the interval $\mathbb{I}$, see \cite{BPnew} or \cite{DJ}.

Next, consider the characteristic function $\lambda'$ on $Q$
determined by
\[
\lambda'(F_1)=\lambda'(F_2)=T^{(\{1\})}.
\]
The corresponding manifold is
\[
X'=(S^1\times\mathbb{I}\times T^2)/\simc=S^1\times
T^{(\{2\})}\times(\mathbb{I}\times T^{(\{1\})}/\simc)\cong
S^1\times S^1\times S^2.
\]

This example shows that in general Betti numbers of $(Q\times
T^2)/\simc$ depend not only on $Q$, but may also depend on the
characteristic function. This is opposite to the situation when
proper faces are acyclic, as was shown in the paper.

\section*{Acknowledgements}
I am grateful to Shintaro Kuroki, from whom I knew about $h'$- and
$h''$-vectors and their possible connections with torus manifolds,
and to Mikiya Masuda for the motivating discussions.


\end{document}